\documentclass{amsart}
\usepackage{latexsym,amssymb,amsthm,amsmath,amscd}

\theoremstyle{plain}
\newtheorem{theorem}{Theorem}[section]
\newtheorem*{Theorem B}{Theorem B}
\newtheorem*{Theorem A}{Theorem A}

\newtheorem{lemma}{Lemma}[section]

\newtheorem{proposition}{Proposition}[section]

\newtheorem{corollary}{Corollary}[section]

\newtheorem{example}{Example}[section]
\numberwithin{equation}{section}

\theoremstyle{remark}
\newtheorem{remark}{Remark}[section]

 \numberwithin{equation}{section}

\def\<{\left < }
\def\>{\right >}
\def\({\left ( }
\def\){\right )}

\def\e{\eqref}

\def\n2{\left[{n\over2}\right]}

\begin{document}

\markboth{B.-Y. Chen}{Anti-holomorphic submanifolds}

\title[Optimal inequalities for anti-holomorphic submanifolds]{Two optimal inequalities for anti-holomorphic submanifolds and their applications}

\author[F. Al-Solamy, B.-Y. Chen and S. Deshmukh]{Falleh R. Al-Solamy, Bang-Yen Chen and Sharief Deshmukh }

\address{Department of Mathematics\\  King Abdulaziz University\\ Jeddah 21589\\ Saudi Arabia}

\email{falleh@hotmail.com}

 \address{Department of Mathematics\\Michigan State University\\619 Red Cedar Road \\ East Lansing, MI 48824--1027, USA}

\email{bychen@math.msu.edu}

 \address{Department of Mathematics\\  King Saud University\\  Riyadh 11451\\ Saudi Arabia}
 \email{shariefd@ksu.edu.sa}

\begin{abstract} The $CR$ $\delta$-invariant for $CR$-submanifolds was introduced by B.-Y. Chen in a recent article \cite{C2012}. In this paper,
we prove two new optimal inequalities for anti-holomorphic submanifolds in complex space forms  involving the $CR$ $\delta$-invariant. Moreover, we obtain some classification results for certain anti-holomorphic submanifolds in complex space forms which satisfy the equality case of either inequality.
\end{abstract}



\thanks{This work is supported by the Deanship of Scientific Research, University of Tabuk. This work was initiated while the second author visited the King Saud University, Saudi Arabia. The second author would like express his many thanks for the hospitality he received during his visit.}

\maketitle

\section{Introduction}

Let $\tilde M$ be a K\"ahler manifold with complex structure $J$ and let $N$ be a Riemannian manifold isometrically immersed in $\tilde M$. For each point $x\in N$, we denote by $\mathcal D_x$ the maximal complex subspace $T_xN\cap J(T_xN)$ of the tangent space $T_xN$ of $N$. If the dimension of $\mathcal D_x$ is the
same for all $x\in N$, then $\{\mathcal D_x, x\in N\}$ define a holomorphic distribution $\mathcal D$ on $N$.
A subspace $\mathcal V$ of $T_{x} N,\, x\in N$ is called {\it totally real} if $J(\mathcal V)$ is a subspace of the normal space $T^{\perp}_{x}N$ at $x$. A submanifold $N$ of a K\"ahler manifold is called a {\it totally real submanifold} if each tangent space of $N$ is totally real.

A submanifold $N$ of a K\"ahler manifold $\tilde M$ is called a {\it $CR$-submanifold} if there exists a totally real distribution $\mathcal D^\perp$ on $N$ whose orthogonal complement is the holomorphic distribution $\mathcal D$ (cf. \cite{B,C1,book2011}), i.e., $$TN={\mathcal D}\oplus {\mathcal D}^\perp,\;\; J\mathcal D^\perp_x\subset T^\perp_x N,\;\; x\in N.$$  
 Throughout this paper, we denote by $h$ the complex rank of the holomorphic distribution $\mathcal D$ and by $p$ the (real) rank of the totally real distribution $\mathcal D^\perp$ for a $CR$-submanifold. A warped product submanifold $N^{T}\times_{f} N^{\perp}$ with warping function $f$ in a K\"ahler manifold $\tilde M$ is called a {\it CR-warped product} if  $N^{T}$ is a holomorphic submanifold and $N^{\perp}$ is a totally real submanifold of $\tilde M$.

It is well-known that the totally real distribution $\mathcal D^{\perp}$ of every $CR$-submanifold of a K\"ahler manifold is an integrable distribution (cf. \cite{C1,c00.1,book2011}). 

In order to provide some answers to an open question concerning minimal immersions proposed by S. S. Chern in the 1960s and to provide some applications of the well-known Nash embedding theorem, the second author introduced in early 1990s the notion of $\delta$-invariants (see \cite{c93,book2011,C2013,V} for details).
For a $CR$-submanifold $N$ of a K\"ahler manifold, he introduced in \cite{C2012} a $\delta$-invariant $\delta(\mathcal D)$, called the {\it $CR$ $\delta$-invariant}, defined by
 \begin{align}\label{1.1} \delta(\mathcal D)(x)=\tau(x)-\tau(\mathcal D_x),\end{align}
where $\tau$ is the scalar curvature of $N$ and $\tau(\mathcal D)$ is the scalar curvature of the holomorphic distribution $\mathcal D$ of $N$ (see \cite{book2011} for details). 
In \cite{C2012}, the second author established a sharp inequality involving the $CR$ $\delta$-invariant $\delta(\mathcal D)$ for anti-holomorphic warped product submanifolds in complex space forms. 

In this paper, we prove two new optimal inequalities involving the $CR$ $\delta$-invariant for arbitrary anti-holomorphic submanifolds in complex space forms. Moreover, we obtain some classification results for anti-holomorphic submanifolds in complex space forms which satisfy the equality case of either inequality.

\section{Preliminaries}

\subsection{Basic definitions and formulas}
 
Let $N$ be a Riemannian $n$-manifold equipped with an inner product $\<\;\, ,\;\>$. Denote by $\nabla$ the Levi-Civita connection of $N$. 

Assume that $N$ is isometrically immersed in a K\"ahler manifold $\tilde M$. Then the formulas of Gauss and Weingarten  are given respectively by (cf. \cite{c1973,book2011})
\begin{align}\label{2.1} &\tilde \nabla_XY=\nabla_X Y+\sigma(X,Y),\\&\label{2.2} \tilde
\nabla_X\xi =-A_\xi X+D_X\xi,\end{align} for vector fields $X$ and $Y$ tangent to $N$ and  $\xi$ normal to $N$, where $\tilde \nabla$ denotes the Levi-Civita connection on $\tilde M$, $\sigma$  is the second fundamental form, $D$ is the normal connection, and $A$ is the  shape operator of $N$. 

The second fundamental form $\sigma$ and the shape operator $A$ are related by
\begin{align}\label{2.3}\<A_\xi X,Y\>=\<\sigma(X,Y), \xi\>,\end{align} where $\<\;\, ,\;\>$
is the inner product on $N$ as well as on $\tilde M$. The mean curvature vector of $N$ is defined by
\begin{align}\label{2.4} \overrightarrow{H}=\(\frac{1}{n}\) {\rm trace} \,\sigma,\;\; n=\dim N.\end{align}
The squared mean curvature $H^{2}$ is given by $H^{2}=\<\right.\hskip-.02in\overrightarrow{H},\overrightarrow{H} \hskip-.02in\left.\>$.

The  {\it equation of Gauss\/} is 
\begin{equation}\begin{aligned} \label{2.5}{
R}(X,Y;Z,W)=&\,\tilde R(X,Y;Z,W)+\<\sigma(X,W),\sigma(Y,Z)\>\\
&-\<\sigma(X,Z),\sigma(Y,W)\>\end{aligned}\end{equation} for vectors
$X,Y,Z,W$ tangent to $N$, where $R$ and $\tilde R$ denote the Riemann curvature tensors of $N$ and $\tilde M$, respectively.

For the second fundamental form $\sigma$, we define its covariant derivative ${\bar\nabla}\sigma$ with respect to the connection on
$TN \oplus T^{\perp}N$ by
\begin{align}\label{2.6}({\bar\nabla}_{X}\sigma)(Y,Z)=D_{X}(\sigma (Y,Z))-\sigma(\nabla_{X}Y,Z) -\sigma(Y,\nabla_{X}Z).\end{align}

The {\it equation of Codazzi\/} is
\begin{align}\label{2.7}({\tilde R}(X,Y)Z)^{\perp}=({\bar\nabla}_{X}\sigma)(Y,Z)-
({\bar\nabla}_{Y}\sigma)(X,Z),\end{align} for vectors $X,Y,Z$ tangent to $N$, where $({\tilde R}(X,Y)Z)^{\perp}$ denotes the normal component of ${\tilde R}(X,Y)Z$. 

\subsection{Real hypersurfaces} A real hypersurface $N$ of a K\"ahler manifold $\tilde M$ is called a {\it Hopf hypersurface}  if $J\xi$ is a principal curvature vector, i.e., an eigenvector of the shape operator $A_\xi$, where
$\xi$ is a unit normal vector of $N$. Obviously, every Hopf hypersurface is mixed totally geodesic.

A real hypersurface $N$ of a K\"ahler manifold $\tilde M$ with $\dim_{\bf C} \tilde M=m$ is called {\it totally real $m$-ruled} if for each point $x\in N$ there exists an $m$-dimensional  totally real totally geodesic submanifold $V_x^{m}$ of $N$ through $x$.

\subsection{Complex space forms}

A K\"ahler manifold of constant holomorphic sectional curvature is called a {\it complex space form}.
Throughout this paper, we denote  a complete, simply-connected (complex) $m$-dimensional complex space form of constant holomorphic sectional curvature $4c$ by $\tilde M^{m}(4c)$. 

It is well-known that $\tilde M^{m}(4c)$ is holomorphically isometric to the complex projective $m$-space $CP^{m}(4c)$, the complex Euclidean $m$-space ${\bf C}^m$, or the complex hyperbolic $m$-space $CH^{m}(4c)$ according to $c>0$, $c=0$, or $c<0$, respectively.

The curvature tensor $\tilde R$ of a complex space form $\tilde M^m(4c)$ satisfies
\begin{equation}\begin{aligned} \label{2.8} \tilde R(U,V,W)=&\, c\{\<V,W\>U-\<X,W\>V+\<JV,W\>JU
\\&\hskip.0in  -\<JU,W\>JV+2\<U,JV\>JW\}.\end{aligned}\end{equation}

\subsection{Anti-holomorphic submanifolds and $CR$-submanifolds} 

  A $CR$-subma\-nifold $N$ of a K\"ahler manifold $\tilde M$ is called {\it anti-holomorphic\/} if we have $$J\mathcal D^\perp_x=T^\perp_x N,\;\; x\in N.$$
  A $CR$-submanifold is called {\it mixed totally geodesic\/} if its second fundamental form $\sigma$ satisfies
 $\sigma(X,Z)=0$ for any $X\in \mathcal D$ and $Z\in \mathcal D^{\perp}$. 
 
 A mixed totally geodesic $CR$-submanifold is called {\it mixed foliate} if its holomorphic distribution $\mathcal D$ is also integrable.
 Moreover, a $CR$-submanifold $N$ is called a {\it $CR$-product} if it is a Riemannian product of a holomorphic submanifold $N^T$ and a totally real submanifold $N^\perp$ of $\tilde M$.

 Obviously, real hypersurfaces of a K\"ahler manifold are exactly anti-holomorphic submanifolds with $p={\rm rank}\,\mathcal D^\perp$=1.

\subsection{$H$-umbilical Lagrangian submanifolds}
  An anti-holomorphic submanifold of a K\"ahler manifold is called {\it Lagrangian} if $\mathcal D=\{0\}$, i.e., $$J(T_xN)=T_x^{\perp}N,\;\; x\in N.$$

A Lagrangian submanifold is said to be {\it $H$-umbilical} if its second fundamental form satisfies the following simple form (cf. \cite{c97}): 
\begin{equation}\begin{aligned} &\label{2.9}
\sigma(e_1,e_1)= \varphi Je_1,\;\; \sigma(e_1,e_j)=\psi Je_j,
\\& \sigma(e_2,e_2)=\cdots = \sigma(e_n,e_n)=\psi
Je_1,\\ & \sigma(e_j,e_k)=0,\; \; j\not=k,
\;\;\;\; j,k=2,\ldots,n,\end{aligned}\end{equation} for some suitable
functions $\varphi$ and $\psi$ with respect to some suitable
orthonormal local frame field $\{e_{1},\ldots,e_{n}\}$.

Since there do not exist umbilical Lagrangian submanifold in K\"ahler manifolds, $H$-umbilical  Lagrangian submanifolds are the simplest Lagrangian submanifolds next to totally geodesic one (cf. \cite{c97}).

\section{Some basic lemmas for $CR$-submanifolds}

We need the following two lemmas from \cite{B,C1} for later use.

\begin{lemma} \label{L:3.1} Let $N$ be a $CR$-submanifold of a K\"ahler manifold $\tilde M$. Then we have:
\begin{enumerate}

\item the totally real distribution $\mathcal D^{\perp}$ is an integrable distribution,

\item $\<\sigma(U,JX),JZ\>=\< \nabla_UX,Z\>$,

\item $A_{JZ}W=A_{JW}Z$,
\end{enumerate}

\noindent for vector field $U$ tangent to $N$, $X,Y$ in $\mathcal D$, and $Z,W$ in $\mathcal D^\perp$.\end{lemma}

\begin{lemma}\label{L:3.2}  Let $N$ be a $CR$-submanifold of a K\"ahler manifold $\tilde M$. Then we have:
\begin{enumerate}

\item the holomorphic distribution $\mathcal D$ is integrable if and only if
\begin{align} \<\sigma(X,JY),JZ\>=\<\sigma(JX,Y),JZ\>\end{align}
holds for any $X,Y\in \mathcal D$ and $Z\in \mathcal D^{\perp}$,

\item the leaves of the totally real distribution $\mathcal D^{\perp}$ are totally geodesic in $N$ if and only if 
\begin{align} \<\sigma(X,Z),JW\>=0\end{align} holds for any $X\in \mathcal D$ and $Z,W\in \mathcal D^{\perp}$.

\end{enumerate}
\end{lemma}

We also recall the following result for later use.

\begin{lemma}\label{L:3.3}  A complex space form $\tilde M^{m}(4c)$ with $c\ne 0$ admits no mixed foliate proper $CR$-submanifolds.
\end{lemma}
Lemma \ref{L:3.3} is due to \cite{BKY} for $c>0$ and due to \cite{CW} for $c<0$.

For mixed foliate $CR$-submanifolds in a complex Euclidean space, we have the following result from \cite{C1}.

\begin{lemma}\label{L:3.4}  Let $N$ be a $CR$-submanifold of ${\bf C}^m$. Then $N$ is mixed foliate if and only if $N$ is a $CR$-product.
\end{lemma}

We also need the following result from \cite[Theorem 4.6]{C1}.

\begin{lemma}\label{L:3.5}   Every $CR$-product
 in a complex Euclidean $m$-space ${\bf C}^m$ is a direct product of a holomorphic
submanifold of a linear complex subspace and a totally real submanifold of another linear complex subspace. 
\end{lemma}

\section{An inequality for  anti-holomorphic submanifolds with $p\geq 2$}

Let $N$ be a $CR$-submanifold of a K\"ahler manifold. Denote by $\mathcal D$ and $\mathcal D^\perp$ the holomorphic distribution and the totally real distribution of $N$ as before.  
For  a $CR$ submanifold $N$, let us choose a local orthonormal frame $\{e_1,\ldots,e_{2h+p}\}$ on $N$ in such way that $e_1,\ldots,e_h,e_{h+1},\ldots,e_{2h}$ are in $\mathcal D$ and $e_{2h+1},\ldots,e_{2h+p}$ are in $\mathcal D^\perp$, where $e_{h+1}=Je_1,\ldots, e_{2h}=Je_h$. 

The  {\it $CR$ $\delta$-invariant}, denoted by $\delta(\mathcal D)$, for  a $CR$-submanifold $N$ with $p={\rm rank}\, {\mathcal D}^{\perp}\geq 1$ is defined by  (see \cite{C2012} for details)
\begin{align}\label{4.1} \delta(\mathcal D)(x)=\tau(x)-\tau(\mathcal D_x),\end{align} 
where $\tau$ and $\tau(\mathcal D)$ denote the scalar curvature of $N$ and the scalar curvature of the holomorphic distribution $\mathcal D\subset TN$, respectively.

Through out this paper, we shall use the following convention on the range of indices {\it unless mentioned otherwise}:
\begin{equation}\begin{aligned} \notag  i,j,k=1,\ldots,2h;\;\; & \alpha,\beta,\gamma=1,\ldots,h,
\\\notag r,s,t=2h+1,\ldots,2h+p;\;\; &A,B,C=1,\ldots,2h+p.\end{aligned}\end{equation}

For a $CR$-submanifold $N$ we define the two {\it partial mean curvature vectors} $\overrightarrow H_{\mathcal D}$ and $\overrightarrow H_{\mathcal D^{\perp}}$ of $N$ by
 \begin{align}\label{4.2} & \overrightarrow H_{\mathcal D}=\frac{1}{2h} \sum_{i=1}^{2h}\sigma(e_{i},e_{i})
,\hskip.2in \overrightarrow H_{\mathcal D^{\perp}}=\frac{1}{p} \sum_{r=2h+1}^{2h+p}\sigma(e_{r},e_{r}).\end{align}

An anti-holomorphic submanifold $N$ of a K\"ahler manifold $\tilde M$ is called {\it minimal} (resp., ${\mathcal D}$-{\it minimal} or ${\mathcal D}^{\perp}$-{\it minimal}\/)
if $H=0$ holds identical (resp., $\overrightarrow H_{\mathcal D}=0$ or $\overrightarrow H_{\mathcal D^{\perp}}=0$ hold identically).

We define the coefficients of the second fundamental form by
$$\sigma^r_{AB}=\<\sigma(e_A,e_B),Je_r\>$$ for $A,B=1,\ldots,2h+p$ and $r=1,\ldots,p$.

For anti-holomorphic submanifolds with $p={\rm rank}\, {\mathcal D}^{\perp}\geq 2$,  we have the following optimal inequality.

\begin{theorem} \label{T:4.1} Let $N$ be an anti-holomorphic submanifold of a complex space form $\tilde M^{h+p}(4c)$ with $h={\rm rank}_{\bf C}\,{\mathcal D}\geq 1$ and  $p={\rm rank}\, {\mathcal D}^{\perp}\geq 2$. Then we have
\begin{align}\label{4.3}  \delta(\mathcal D)   \leq  \frac{(p-1)(2h+p)^2}{2(p+2)}H^2+\dfrac{p}{2}(4h+p-1)c.
\end{align}

The equality sign of \e{4.3} holds identically if and only if the following three conditions are satisfied:
\begin{enumerate}

\item[\rm (a)] $N$ is $\mathcal D$-minimal,  i.e.,  $\overrightarrow H_{\mathcal D}=0$,

\item[\rm (b)] $N$ is mixed totally geodesic, and

\item[\rm (c)]  there exist an orthonormal frame  $\{e_{2h+1},\ldots,e_n\}$ of $\mathcal D^\perp$ such that the second fundamental $\sigma$ of $N$ satisfies 
\begin{equation}\begin{aligned}
\label{4.4} &\begin{cases}\sigma^r_{rr}=3\sigma^r_{ss}, \hskip.2in {\rm for}\;\;  2h+1\leq r\ne s\leq 2h+p,\\
\sigma^r_{st}=0,\hskip .4in {\rm for \; distinct}\; \; r,s,t\in \{2h+1,\ldots,2h+p\}.\end{cases}\end{aligned}\end{equation}
\end{enumerate}
\end{theorem}
\begin{proof} Let $N$ be an anti-holomorphic submanifold in a complex space form $\tilde M^{h+p}(4c)$. Let us choose an orthonormal frame $\{e_1,\ldots,e_{2h+p}\}$  on $N$ as above.

It follows from the equation of Gauss and the definition of $CR$ $\delta$-invariant that  $\delta(\mathcal D)$ satisfies
\begin{equation}\begin{aligned}\label{4.5} \delta(\mathcal D)=&\sum_{i=1}^{2h}\sum_{r=2h+1}^{2h+p} K(e_i,e_r)+\sum_{2h+1\leq r\ne s\leq 2h+p}\frac{1}{2} K(e_r,e_s)
\\ = &\sum_{i=1}^{2h}\sum_{r=2h+1}^{2h+p}\<\sigma(e_i,e_i),\sigma(e_r,e_r)\>+\sum_{r,s=2h+1}^{2h+p}\frac{1}{2}\<\sigma(e_r,e_r),\sigma(e_s,e_s)\>
\\& -\sum_{i=1}^{2h}\sum_{r=2h+1}^{2h+p} ||\sigma(e_i,e_r)||^2-\sum_{r,s=2h+1}^{2h+p} \frac{1}{2}||\sigma(e_r,e_s)||^2\\&+\frac{p}{2}(4h+p-1)c.\end{aligned}\end{equation}

On the other hand, we have
\begin{equation}\begin{aligned} \label{4.6} &\sum_{i=1}^{2h}\sum_{r=2h+1}^{2h+p}\! \<\sigma(e_i,e_i),\sigma(e_r,e_r)\>+\sum_{r,s=2h+1}^{2h+p}\!\frac{1}{2} \<\sigma(e_r,e_r),\sigma(e_s,e_s)\>\! \\&\hskip.8in -\sum_{r,s=2h+1}^{2h+p} \frac{1}{2}||\sigma(e_r,e_s)||^2\\&\hskip.1in =\frac{(2h+p)^2}{2}H^2-2h^{2}|\overrightarrow H_{\mathcal D}|^{2} -\frac{1}{2}||\sigma_{\mathcal D^{\perp}}||^2,\end{aligned}\end{equation}
  where  $||\sigma_{\mathcal D^{\perp}}||^2$  is defined by
 \begin{align}\label{4.7} &||\sigma_\perp||^2=\sum_{r,s=2h+1}^{2h+p}||\sigma(e_r,e_s)||^2.\end{align}
By combining \eqref{4.5} and \eqref{4.6} we find
\begin{equation}\begin{aligned}\label{4.8} \delta(\mathcal D)&=\frac{(2h+p)^2}{2}H^2+\frac{p}{2}(4h+p-1)c -2h^{2}|\overrightarrow H_{\mathcal D}|^{2}\\&\hskip.3in -\sum_{i=1}^{2h}\sum_{r=2h+1}^{2h+p} ||\sigma(e_i,e_r)||^2 -\frac{1}{2}||\sigma_{\mathcal D^{\perp}}||^2.\end{aligned}\end{equation}

It follows from statement (2) of  Lemma \ref{L:3.1} the coefficients of the second fundamental form satisfy \begin{align}\label{4.9}\sigma^r_{st}=\sigma^s_{rt}=\sigma^t_{rs}.\end{align}

 We find from \eqref{4.2}, \eqref{4.7} and \eqref{4.9}  that
\begin{equation}\begin{aligned}\label{4.10}& (p+2)||\sigma_{\mathcal D^{\perp}}||^2-3p^{2}|H_{\mathcal D^{\perp}}|^{2}
=(p-1)\sum_{r=2h+1}^{2h+p}\Bigg(\sum_{s=2h+1}^{2h+p} \sigma^r_{ss}\Bigg)^2\\&\hskip.2in +\!\sum_{2h+1\leq r\ne s\leq 2h+p}\!3(p+1)(\sigma^r_{ss})^2  + \sum_{2h+1\leq r<s<t\leq 2h+p}\!6(p+2)(\sigma^r_{st})^2
\\&\hskip.5in+\sum_{r=2h+1}^{2h+p}\,\sum_{2h+1\leq s<t\leq 2h+p} 2(p+2)\sigma^r_{ss}\sigma^r_{tt}
\\ & \hskip.1in =\sum_{r=2h+1}^{2h+p}(p-1)(\sigma^r_{rr})^2+\sum_{2h+1\leq r\ne s\leq 2h+p}3(p+1)(\sigma^r_{ss})^2 \\&\hskip.2in+\sum_{2h+1\leq r<s<t\leq 2h+p}6(p+2)(\sigma^r_{st})^2
 -  \sum_{r=2h+1}^{2h+p}\,\sum_{2h+1\leq s<t \leq 2h+p}\! 6\sigma^r_{ss}\sigma^r_{tt} 
\\& \hskip.1in = \sum_{2h+1\leq r<s<t\leq 2h+p}\!6(p+2) (\sigma^r_{st})^2+\sum_{2h+1\leq s\ne r\leq 2h+p} (\sigma^r_{rr}-3\sigma^r_{ss})^2\\&\hskip.7in  +\sum_{r\ne s,t}\; \sum_{2h+1\leq s<t\leq 2h+p}3(\sigma^r_{ss}-\sigma^r_{tt})^2
\\&\hskip.1in \geq 0.\end{aligned}\end{equation}
Thus we get
\begin{align}\label{4.11} ||\sigma_{\mathcal D^{\perp}}||^2\geq \frac{3p^{2}}{p+2}|H_{\mathcal D^{\perp}}|^{2},
\end{align}
with equality holding if and only if 
\begin{equation}\begin{aligned}\label{4.12} &\sigma^r_{rr}=3\sigma^r_{ss}, \hskip.2in {\rm for}\;\;  2h+1\leq r\ne s\leq 2h+p,\\&
\sigma^r_{st}=0,\hskip .4in {\rm for \; distinct}\; \; r,s,t\in \{2h+1,\ldots,2h+p\}.\end{aligned}\end{equation}

Now, by combining \eqref{4.8} and  \eqref{4.11}, we obtain
\begin{equation}\begin{aligned}\label{4.13} & \frac{(2h+p)^2}{2}H^2+\frac{p}{2}(4h+p-1)c- \delta(\mathcal D)
\\&\hskip.2in \geq  2h^{2}|\overrightarrow H_{\mathcal D}|^{2} +\sum_{i=1}^{2h}\sum_{r=2h+1}^{2h+p}||\sigma(e_i,e_r)||^2 +\frac{3p^{2}}{2(p+2)}|H_{\mathcal D^{\perp}}|^{2}
\\& \hskip.2in =  \frac{3}{2(p+2)}\Bigg\{(2h+p)^{2}H^{2}-4h^{2}|\overrightarrow H_{\mathcal D}|^{2}-2\sum_{i=1}^{2h}\sum_{r=2h+1}^{2h+p}||\sigma(e_i,e_r)||^2\Bigg\}
\\&\hskip.5in +2h^{2}|\overrightarrow H_{\mathcal D}|^{2} +\sum_{i=1}^{2h}\sum_{r=2h+1}^{2h+p}||\sigma(e_i,e_r)||^2 
\\& \hskip.2in =  \frac{3(2h+p)^{2}}{2(p+2)}H^{2}+\frac{2h^{2}(p-1)}{p+2}|\overrightarrow H_{\mathcal D}|^{2}+\frac{p-1}{p+2}\sum_{i=1}^{2h}\sum_{r=2h+1}^{2h+p}||\sigma(e_i,e_r)||^2
\\&\hskip.2in \geq  \frac{3(2h+p)^{2}}{2(p+2)}H^{2}
.\end{aligned}\end{equation}
It is obvious that the equality of the last inequality in \eqref{4.13} holds if and only if $N$ is $\mathcal D$-minimal  and  mixed totally geodesic.
Consequently, we may obtain inequality \eqref{4.3}  from \eqref{4.13}.

It is straightforward to verify that the equality sign of \eqref{4.3} holds identically if and only if conditions (a), (b) and (c) of Theorem \ref{T:4.1} are satisfied. 
\end{proof}

\section{Anti-holomorphic submanifolds with $p\geq 2$ satisfying equality}

First, we give the following example satisfying the equality case of \e{4.3}.

\begin{example} \label{E:4.1} {\rm Let $w:S^p(1)\rightarrow {\bf C}^p, \,p\geq 2,$ be the map of the unit $p$-sphere $S^p(1)$ into ${\bf C}^{p}$ defined by
$$w(y_0,y_1,\ldots,y_p)={{1+{\rm i}y_0}\over {1+y_0^2}}( y_1,\ldots,y_p),\quad y_0^2+y_1^2+\ldots+y_p^2=1.$$
The map $w$ is a (non-isometric) Lagrangian immersion  with one self-intersection point. This immersion is called the {\it Whitney $p$-sphere.}  It is well-known that Whitney spheres are the only $H$-umbilical Lagrangian submanifolds of the complex Euclidean spaces satisfying $\alpha=3\beta\ne 0$ in \e{2.8} (see for instance, \cite{BCM,book2011}).

Consider the product immersion: $$\phi:{\bf C}^{h}\times S^{p}(1)\to {\bf C}^{h}\oplus {\bf C}^{p}={\bf C}^{h+p}$$ defined by
\begin{align}\label{5.1} \phi(z,x)=(z,w(x)),\;\; \forall z\in {\bf C}^{h},\;\; \forall x\in S^{p}(1).\end{align}
It is straight-forward to verify that $\phi$ is an anti-holomorphic isometric immersion which satisfies the equality sign of \e{4.3} identically.}\end{example}

In this section we provide the following two classification theorems for anti-holomorphic submanifolds satisfying the equality case of \e{4.3} identically.

\begin{theorem}\label{T:5.1} Let $N$ be an anti-holomorphic submanifold of a complex space form $\tilde M^{h+p}(4c)$ with $h={\rm rank}_{\bf C}\,{\mathcal D}\geq 1$ and  $p={\rm rank}\, {\mathcal D}^{\perp}\geq 2$. If $N$ satisfies the equality case of  \e{4.3} identically and if the holomorphic distribution $\mathcal D$ is integrable, then $c=0$ so that $\tilde M^{h+p}(4c)={\bf C}^{h+p}$. Moreover,  either 

\begin{enumerate}
\item[{\rm (i)}] $N$ is a totally geodesic anti-holomorphic submanifold of ${\bf C}^{h+p}$ or, 

\item[{\rm (ii)}] up  to dilations and rigid motions of ${\bf C}^{h+p}$, $N$ is  given by an open portion of the following product immersion:
\begin{align}\notag \hskip.2in \phi:{\bf C}^h\times S^{p}(1)\to {\bf C}^{h+p};\hskip.1in (z,x)\mapsto (z,w(x)),\;\; z\in {\bf C}^h,\;   x\in S^{p}(1),\end{align}where $w:S^{p}(1)\to {\bf C}^{p}$ is the Whitney $p$-sphere defined in Example \ref{E:4.1}.
\end{enumerate}
\end{theorem}
\begin{proof} Assume that $N$ is an anti-holomorphic submanifold of a complex space form $\tilde M^{h+p}(4c)$ with $h={\rm rank}_{\bf C}\,{\mathcal D}\geq 1$ and  $p={\rm rank}\, {\mathcal D}^{\perp}\geq 2$. If $N$ satisfies the equality case of  \e{4.3} and if the holomorphic distribution $\mathcal D$ is integrable, then it follows from Theorem \ref{T:4.1} that $N$ is mixed foliate. Hence Lemma \ref{L:3.3} implies that $c=0$. Therefore, according to Lemma \ref{L:3.4}, $N$ is a $CR$-product. Hence,  $N$ is locally a $CR$-product given by
$${\bf C}^h\times N^{\perp}\subset {\bf C}^h\times {\bf C}^{p},$$ where ${\bf C}^h$ is a complex Euclidean $h$-subspace and $N^{\perp}$ is a Lagrangian submanifold of ${\bf C}^p$.  Consequently, condition (c) of Theorem \ref{T:4.1} implies that $N^{\perp}$ is a  Lagrangian $H$-umbilical submanifold  in ${\bf C}^{p}$  whose second fundamental form satisfying 
\begin{equation}\begin{aligned} &\label{5.2}
\sigma(e_{2h+1},e_{2h+1})= 3\lambda Je_{2h+1},\;\; \sigma(e_{2h+1},e_s)=\lambda Je_s,\\& \sigma(e_{2h+2},e_{2h+2})=\cdots = \sigma(e_{2h+p},e_{2h+p})=\lambda Je_{2h+1},\\ & \sigma(e_r,e_s)=0,\; \; 2h+2\leq r\ne s\leq 2h+p,
\end{aligned}\end{equation} for some suitable
function $\lambda$ with respect to some suitable
orthonormal local frame field $\{e_{2h+1},\ldots,e_{2h+p}\}$ of $TN^{\perp}$. 

If $\lambda=0$, then $N^\perp$ is an open portion of a totally geodesic totally real $p$-plane in ${\bf C}^p$. Hence, in this case $N$ is a totally geodesic anti-holomorphic submanifold. 

If $\lambda\ne 0$, it follows from \e{5.2} that, up to dilations and rigid motions, $N^{\perp}$ is an open part of the Whitney $p$-sphere in ${\bf C}^{p}$ (cf. \cite{BCM,book2011}). Therefore, up  to dilations and rigid motions of ${\bf C}^{h+p}$ the anti-holomorphic submanifold is locally given by the product immersion:
\begin{align}\phi:{\bf C}^h\times S^{p}(1)\to {\bf C}^{h+p};\hskip.2in (z,x)\mapsto (z,w(x)),\end{align} for $z\in {\bf C}^h$ and $ x\in S^{p}(1)$, where $w:S^{p}(1)\to {\bf C}^{p}$ is the Whitney $p$-sphere.

The converse is easy to verify. \end{proof}

\begin{theorem} \label{T:5.2} Let $N$ be an anti-holomorphic submanifold in a complex space form  $\tilde M^{1+p}(4c)$ with $h={\rm rank}_{\bf C}\,\mathcal D=1$ and $p={\rm rank}\,\mathcal D^{\perp}\geq 2$. Then we have
\begin{align}\label{5.4}  \delta(\mathcal D) \leq  \dfrac{(p-1)(p+2)^2}{2(p+2)}H^2+\dfrac{p}{2}(p+3)c.\end{align}

The equality case of \eqref{5.4} holds identically if and only if $c=0$ and either 

\begin{enumerate}
\item[{\rm (i)}] $N$ is a totally geodesic anti-holomorphic submanifold of ${\bf C}^{h+p}$ or, 

\item[{\rm (ii)}] up  to dilations and rigid motions, $N$ is given by an open portion of the following product immersion:
\begin{align}\notag\phi:{\bf C}\times S^{p}(1)\to {\bf C}^{1+p};\hskip.2in (z,x)\mapsto (z,w(x)),\;\; z\in {\bf C},\;   x\in S^{p}(1),\end{align}where $w:S^{p}(1)\to {\bf C}^{p}$ is the Whitney $p$-sphere.
\end{enumerate} \end{theorem}
\begin{proof} Let $N$ be an anti-holomorphic submanifold in a complex space form  $\tilde M^{1+p}(4c)$. Then we have inequality \e{5.4} from inequality \e{4.3}. 

 Assume that $N$ satisfies the equality case of \e{5.4} identically. Then Theorem \ref{T:4.1} implies that $N$ satisfies conditions (a), (b) and (c) of Theorem \ref{T:4.1}.

By condition (a), $N$ is $\mathcal D$-minimal. Thus we find
\begin{align} \label{5.5} \sigma(Je_{1},Je_{1})=-\sigma(e_{1},e_{1})\end{align}
 for any unit vector $e_{1}\in \mathcal D$. 
It is direct to verify from \e{5.5} and polarization that the second fundamental form satisfies the following condition:
\begin{align}\notag \sigma(X,JY)=\sigma(JX,Y),\;\; \forall X,Y\in \mathcal D.\end{align}
Therefore, according to Lemma \ref{L:3.2}(1), we may conclude that $\mathcal D$ is integrable.  Consequently, we obtain Theorem \ref{T:5.2} from Theorem \ref{T:5.1}.
\end{proof}

\section{An optimal inequality for real hypersurfaces}

Clearly, anti-holomorphic submanifolds with ${\rm rank}\, {\mathcal D}^{\perp}=1$ are nothing but real hypersurfaces. The Ricci tensor $Ric$ of real hypersurfaces in complex space forms have been studied in \cite{C02,D,S} among others. 

In the following,  a Hopf hypersurface $N$ is called {\it special} if  $J\xi$ is an eigenvector of $A_\xi$ with eigenvalue 0, i.e., $A_\xi(J\xi)=0$, where  $\xi$ is a unit normal vector field.

For real hypersurfaces, we have the following.

\begin{theorem} \label{T:6.1} If $N$ is a real hypersurface of a complex space form $\tilde M^{h+1}(4c)$, then the Ricci  tensor $Ric$ of $N$ satisfies
\begin{align} \label{6.1} Ric(J\xi,J\xi)\leq \frac{(2h+1)^2}{2}H^2+2hc.\end{align}
where $\xi$ is a unit normal vector field of $N$ in $\tilde M^{h+1}(4c)$.

The equality sign of \e{6.1} holds identically if and only if $N$ is  a minimal special Hopf hypersurface.
\end{theorem}
\begin{proof} Let $N$ be a real hypersurface of a complex space form $\tilde M^{h+1}(4c)$. Then it follows from the definition of $\delta(\mathcal D)$ that 
\begin{align}\label{6.2}\delta(\mathcal D)= Ric(J\xi,J\xi).\end{align}

Let us choose an orthonormal frame 
$\{e_1,\ldots,e_h,e_{h+1}=Je_1,\ldots, e_{2h}=Je_h\}$ 
for the holomorphic distribution $\mathcal D$ and let $e_{2h+1}$ be a unit vector field in ${\mathcal D}^\perp$. 

We put
\begin{align}\label{6.3} \sigma_{a,b}=\<\sigma(e_a,e_b),Je_{2h+1}\>,\;\;a,b=1,\ldots,2h+1. \end{align}

Let us define the connection forms by \begin{equation} \begin{aligned} \label{6.4}&\nabla_X e_i=\sum_{j=1}^{2h} \omega_i^j(X)e_j  +\omega_i^{2h+1}(X)e_{2h+1},
\\& \nabla_X e_{2h+1}=\sum_{j=1}^{2h} \omega_{2h+1}^j(X)e_j ,\end{aligned}\end{equation}
for $i=1,\ldots,2h$.
It follows  from \eqref{4.1} and the equation of Gauss that  
\begin{equation}\begin{aligned}\label{6.5} \delta(\mathcal D) = &\sum_{i=1}^{2h} \sigma_{i,i}\sigma_{2h+1,2h+1} -\sum_{i=1}^{2h}(\sigma_{i,2h+1})^2+2hc.\end{aligned}\end{equation}

On the other hand, we have
\begin{equation}\begin{aligned} \label{6.6} \sum_{i=1}^{2h} \sigma_{i,i}\sigma_{2h+1,2h+1} =&\frac{(2h+1)^2}{2} H^2 - \frac{1}{2}(\sigma_{2h+1, 2h+1})^2
  -2h^{2}|\overrightarrow H_{\mathcal D}|^{2}.\end{aligned}\end{equation}
 By combining \eqref{6.6} and \eqref{6.6} we obtain
\begin{equation}\begin{aligned}\label{6.7} \delta(\mathcal D)&=\frac{(2h+1)^2}{2}H^2+2hc -2h^{2}|\overrightarrow H_{\mathcal D}|^{2}  -\frac{1}{2}(\sigma_{2h+1, 2h+1})^2
\\&\hskip.3in - \sum_{i=1}^{2h}(\sigma_{i,2h+1})^2 \\& \leq \frac{(2h+1)^2}{2}H^2+2hc.\end{aligned}\end{equation}
It follows from \e{6.7} and Lemma \ref{L:3.2}(2) that
the equality sign of inequality \e{6.1} holds identically if and only if the following two statements hold:
\begin{enumerate}
\item[(i)] $N$ is a special Hopf hypersurface and 

\item[(ii)]  $N$ is $\mathcal D$-minimal in $\tilde M^{h+1}(4c)$.
\end{enumerate}
Obviously, conditions (i) and(ii) imply that $N$ is a minimal real hypersurface of   $\tilde M^{h+1}(4c)$.

The converse is easy to verify.
\end{proof}

The following corollary follows easily from Theorem \ref{T:6.1}.

\begin{corollary} Let  $N$ be a real hypersurface of a complex space form $\tilde M^{h+1}(4c)$.
If $N$ satisfies the equality case of \e{6.1} identically, then the holomorphic distribution of $N$ is non-integrable, unless $c=0$ and $N$ is totally geodesic.
\end{corollary}
\begin{proof} Under the hypothesis, if  $N$ satisfies the equality case of \e{6.1} identically and if  the holomorphic distribution $\mathcal D$ is integrable, then Theorem \ref{T:6.1} implies that $N$ is mixed foliate. So, it follows from Lemma \ref{L:3.3}  and Lemma \ref{L:3.4} that $c=0$ and $N$ is a $CR$-product of a complex $h$-subspace in ${\bf C}^h$ and an open portion of line in ${\bf C}$. Consequently, $N$ must be totally geodesic.
\end{proof}

\section{Some applications of Theorem \ref{T:6.1} }

We need the following lemma.

\begin{lemma} \label{L:7.1} Let $N$ be a special Hopf hypersurface of a complex space form $\tilde M^{h+1}(4c)$. Then there exist  an orthonormal frame $\{e_1,\ldots,e_h,e_{h+1}=Je_1,\ldots,e_{2h}=Je_{h}\}$ of $\mathcal D$ and an integer $k\leq h$  such that
\begin{equation} \begin{aligned}\label{7.1} & \sigma(e_\alpha,e_\beta)=\lambda_\alpha \delta_{\alpha\beta}\xi,\; \;\sigma(e_{h+\alpha},e_{h+\beta})=\mu_\alpha \delta_{\alpha\beta}\xi, \;\; 1\leq \alpha,\beta \leq k;
\\& \sigma(e_a,e_b)=0,\;\; otherwise,
\end{aligned}\end{equation}
with $\lambda_\alpha\mu_\alpha=c$, where  $\kappa_1,\ldots,\kappa_k$ are nonzero functions.
\end{lemma}
\begin{proof} Let $N$ be a special Hopf hypersurface of  $\tilde M^{h+1}(4c)$ and let $e_{2h+1}$ be a unit vector field in $\mathcal D^\perp$.
Then  $\xi=Je_{2h+1}$ is a unit normal vector field. Thus we have
\begin{align} \label{7.2} \sigma(U,e_{2h+1})=0,\;\; U\in TN. \end{align}
For an eigenvector $X$ of $A_\xi$ with eigenvalue $\kappa\ne 0$,  we may choose an orthonormal frame 
$\{e_1,\ldots,e_h,e_{h+1}=Je_1,\ldots, e_{2h}=Je_h\}$ with $e_1=X$.
Hence we find
\begin{align}\label{7.3} A_\xi(e_1)=\kappa e_1.\end{align}
From \e{6.3}, \e{6.4} and Lemma \ref{L:3.1}(2) we derive that
\begin{equation} \begin{aligned} \label{7.4}  &\omega_\alpha^{2h+1}(e_\beta)=\sigma_{\alpha+h,\beta },\;\;  \omega_{h+\alpha}^{2h+1}(e_\beta)=-\sigma_{\alpha,\beta},
\\& \omega_\alpha^{2h+1}(e_{h+\beta})=\sigma_{h+\alpha,h+\beta},\;\;  \omega_{h+\alpha}^{2h+1}(e_{h+\beta})=-\sigma_{\alpha,h+\beta}.\end{aligned}\end{equation}
It follows from \e{2.8} that
\begin{align} \label{7.5}&(\tilde R(e_\alpha,e_{h+\beta})e_{2h+1})^\perp=-2c\delta_{\alpha\beta} Je_{2h+1}.
\end{align}

On the other hand, we find from \e{7.2}, \e{7.4} and the equation of Codazzi that
\begin{equation} \begin{aligned}\label{7.6} (\tilde R(e_\alpha,e_{h+\beta})e_{2h+1})^\perp &\,= (\bar\nabla_{e_\alpha} \sigma)(e_{h+\beta},e_{2h+1}) -(\bar\nabla_{e_{h+\beta}} \sigma)(e_{\alpha},e_{2h+1})
\\&\,  = 2 \sum_{\gamma=1}^h(\sigma_{\alpha,h+\gamma}\sigma_{\gamma,h+\beta}-\sigma_{\alpha,\gamma}\sigma_{h+\beta,h+\gamma})Je_{2h+1}.
\end{aligned}\end{equation}
By combining \e{7.5} and \e{7.6} we find
\begin{align}\label{7.7} \sum_{\gamma=1}^h(\sigma_{\alpha,\gamma}\sigma_{h+\beta,h+\gamma}-\sigma_{\alpha,h+\gamma}\sigma_{\gamma,h+\beta})=c\delta_{\alpha\beta},\;\; 1\leq\alpha,\beta\leq h. \end{align}
Also, it follows from $(\tilde R(e_\beta,e_{\alpha})e_{2h+1})^\perp
 = \sigma(e_{\alpha},\nabla_{e_{\beta}}e_{2h+1})- \sigma(e_{\beta}, \nabla_{e_\alpha} e_{2h+1})$ that
\begin{equation} \label{7.8}  \sum_{\gamma=1}^h(\sigma_{\alpha,h+\gamma}\sigma_{\beta,\gamma}-\sigma_{\alpha,\gamma}\sigma_{\beta,h+\gamma})=0.
\end{equation}
Condition \e{7.3} gives
\begin{align} \label{7.9} \sigma_{11}=\kappa\ne 0,\;\; \sigma_{1a}=0,\;\; otherwise.
\end{align}
Now, by combining \e{7.7}, \e{7.8} and \e{7.9} we obtain
\begin{align} \notag \kappa\sigma_{1^*1^*}=c\;\;{\rm and}\;\; \sigma_{1^*a}=0,\;\; {\rm for}\;\; a=1,\ldots,h,2^*,\ldots,h^*,
\end{align}
which implies that $JX=e_{1^*}$ is an eigenvector of $A_\xi$ with eigenvalue $c/\kappa$. By applying this fact, we conclude the lemma.
\end{proof}

\begin{remark} Lemma \ref{L:7.1} is due to \cite{M} and \cite{Ber} for $c>0$ and $c<0$, respectively,
\end{remark}

Lemma \ref{L:7.1} implies the following two lemmas.

\begin{lemma} \label{L:7.2} If $N$ is a special Hopf hypersurface of ${\bf C}^{h+1}$, then there exists an orthonormal frame $\{e_1,\ldots,e_h,e_{h+1}=Je_1,\ldots,e_{2h}=Je_{h}\}$ of $\mathcal D$ and an integer $k\leq h$ such that
\begin{equation} \begin{aligned}\label{7.10} & \sigma(e_\alpha,e_\beta)=\lambda_\alpha \delta_{\alpha\beta}\xi,\; \; \sigma(e_a,e_b)=0,\;\; otherwise,
\end{aligned}\end{equation}
for $1\leq \alpha,\beta\leq k$, where $\lambda_1,\ldots,\lambda_{k}$ are nonzero functions.
\end{lemma}
\begin{proof} Under the hypothesis, Lemma \ref{L:7.1} implies that there is an orthonormal frame $\{e_1,\ldots,e_h,e_{h+1}=Je_1,\ldots,e_{2h}=Je_{h}\}$ of $\mathcal D$ such that
\begin{equation} \begin{aligned} & \sigma(e_\alpha,e_\beta)=\lambda_\alpha \delta_{\alpha\beta}\xi,\; \;\sigma(e_{h+\gamma},e_{h+\eta})=\mu_\gamma \delta_{\gamma\eta}\xi, \;\; 
\\& \sigma(e_a,e_b)=0,\;\; otherwise,
\\& 1\leq \alpha,\beta \leq n_1;\;\; n_1+1\leq \gamma,\eta\leq n_1+n_2,
\end{aligned}\end{equation}
 where $n_1,n_2$ are integers  and $\lambda_1,\ldots,\lambda_{n_1+n_2}$ are functions. Thus, after replacing $$e_{n_1+1},\ldots, e_{n_1+n_2},Je_{n_1+1},\ldots, Je_{n_1+n_2}\,\mu_{n_1+1},\ldots,\mu_{n_1+n_2}$$ by $Je_{n_1+1},\ldots, Je_{n_1+n_2},-e_{n_1+1},\ldots, -e_{n_1+n_2},\lambda_{n_1+1},\ldots,\lambda_{n_1+n_2},$ 
respectively,  we obtain \e{7.10}.
\end{proof}

\begin{lemma} \label{L:7.3} Let $N$ be a special Hopf hypersurface of  $CP^{h+1}(4)\; ($resp., $CH^{h+1}(-4))$. Then there exists an orthonormal frame $\{e_1,\ldots,e_h,e_{h+1}=Je_1,\ldots,e_{2h}=Je_{h}\}$ of the holomorphic distribution $\mathcal D$ such that
\begin{equation} \begin{aligned}\notag & \sigma(e_\alpha,e_\beta)=\lambda_\alpha \delta_{\alpha\beta}\xi,\; \;  \sigma(e_{h+\alpha},e_{h+\beta})=\frac{ \delta_{\alpha\beta}}{\lambda_\alpha}\xi,\;\; \sigma(e_a,e_b)=0\;\, otherwise,\\&(resp.,\; \sigma(e_\alpha,e_\beta)=\lambda_\alpha \delta_{\alpha\beta}\xi,\;   \sigma(e_{h+\alpha},e_{h+\beta})=-\frac{ \delta_{\alpha\beta}}{\lambda_\alpha}\xi,\; \sigma(e_a,e_b)=0,\; otherwise),
\end{aligned}\end{equation}
for $1\leq \alpha,\beta\leq h$, where $\lambda_1,\ldots,\lambda_{h}$ are nowhere zero functions.
\end{lemma}

By applying Theorem \ref{T:6.1} and Lemma \ref{L:7.1}, we have the following.

\begin{theorem} \label{T:7.1} If $N$ is a real hypersurface of $\tilde M^{2}(4c)$, then  we have
\begin{align} \label{7.12} Ric(J\xi,J\xi)\leq \frac{9}{2}H^2+2c.\end{align}

The equality sign of \e{7.12} holds identically if and only if  $c=0$ and $N$ is totally geodesic.
\end{theorem}
\begin{proof} Let $N$ be a real hypersurface of a complex space form $\tilde M^{2}(4c)$. Then we obtain \e{7.12} from \e{6.1}.
Assume that $N$ satisfies the equality case of \e{7.12} identically. Then Theorem \ref{T:6.1} implies that $N$ is a minimal special Hopf hypersurface. Therefore, by Lemma \ref{L:7.1} there exists a unit vector field  $e_1$ in $\mathcal D$ such that \begin{equation} \begin{aligned}\label{7.13} &\sigma(e_1,e_1)=\lambda Je_3,\; \sigma(e_{2},e_{2})=-\lambda Je_3,\; \\& \sigma(e_1,e_{2})=\sigma(e_2,e_3)=0,\; a=1,2,3 \end{aligned}\end{equation}
 for some function $\lambda$.
 It follows from  \e{7.13}  and Lemma \ref{L:3.1}(2) that
 \begin{align}\label{7.14}&\omega^3_1(e_1)=\omega^3_2(e_2)=0,
 \;\;  \omega_3^2 (e_1)=\omega_3^1(e_2)=\lambda.\end{align}

On the other hand, we find  from 
$(\tilde R(e_1,e_2)e_3)^\perp=(\bar\nabla_{e_1} \sigma)(e_2,e_3)- (\bar\nabla_{e_2} \sigma)(e_1,e_3)$, \e{2.8} and \e{7.4}  that
$ -2c=\lambda (\omega^1_3(e_2)+\omega_3^2(e_1))$.
 Combining this with \e{7.14} gives \begin{align}\label{7.16} c=-\lambda^2\leq 0.\end{align}
\vskip.1in

If $c=0$,  \e{7.16} implies that $\lambda=0$. Thus $N$ is a totally geodesic hypersurface. 

If $c<0$,  it follows from \e{7.16} that $\lambda$ is a nonzero constant. Thus, $N$ is a minimal Hopf hypersurface of $CH^2(-\lambda^2)$ with three constant principal curvatures $0,\lambda,-\lambda$. But this is impossible according to Theorem 1 of \cite{Ber}.

The converse is easy to verify.
\end{proof}

For real hypersurfaces in ${\bf C}^3$, we have the following.

\begin{theorem}\label{T:7.2} Let $N$ be a real hypersurface of ${\bf C}^3$. We have
\begin{align} \label{7.17} Ric(J\xi,J\xi)\leq \frac{25}{2}H^2.\end{align}

If the equality case of \e{7.17} holds identically, then $N$ is a  totally real 3-ruled minimal submanifold of ${\bf C}^3$.
\end{theorem} 
\begin{proof} Let $N$ be a real hypersurface of ${\bf C}^3$. Then we find inequality \e{7.17} from \e{6.1} of Theorem \ref{T:6.1}.

Assume that $N$ satisfies the equality case of \e{7.17} identically. Then it follows from Theorem \ref{T:6.1} and Lemma \ref{L:7.2} that  there exists an orthonormal local frame $\{e_1,e_2,e_3=Je_1,e_4=Je_{2},e_5\}$ on $N$ such that
\begin{equation}\begin{aligned}\label{7.18.1} & \sigma(e_1,e_1)=\lambda \xi,\; \;  \sigma(e_{2},e_{2})=- \lambda\xi,\;\; \\&\sigma(e_a,e_b)=0\;\, otherwise ,
\end{aligned}\end{equation} for some function $\lambda$. 

Let us put $W=\{x\in N:\lambda(x)\ne 0\}$, which  is an open subset of $W$. 

{\it Case} (a): $W=\emptyset$. In this case, $N$ is a totally geodesic hypersurface. In particular, $N$ is a totally real 3-ruled minimal submanifold of ${\bf C}^3$.

{\it Case} (b): $W\ne \emptyset$. If we put  ${\mathcal D}_1={\rm Span}\,\{e_1,e_2\}$ and $ {\mathcal D}_2={\rm Span}\,\{e_3,e_4,e_5\},$ then we find from \e{7.18.1} that
\begin{align}\label{7.19.1} \sigma({\mathcal D}_2,TN)=\{0\},\;\;i.e., \; A_\xi V=0,\;\; \forall V\in \mathcal D_2.\end{align}
Thus, after applying \e{7.19.1} and the Codazzi equation $$(\bar\nabla_U\sigma)(V,X)=(\bar\nabla_V\sigma)(U,X),\;\; U,V\in {\mathcal D}_2,\;\; X\in {\mathcal D}_1,$$
we obtain $\sigma([U,V],X)=0$. Therefore, it follows from\e{7.18.1} and $\lambda\ne 0$ that $\mathcal D_2$ is an integrable distribution.

Also,  from \e{7.18.1} we derive that
\begin{equation}\begin{aligned}\notag  \lambda\<\nabla_U V,e_1\>&\,=\<\nabla_U V,A_\xi e_1\>
\\&=-\<V,(\nabla_UA_\xi)e_1\>-\< V,A_\xi(\nabla_U e_1)\>
\\&=-\<V,(\nabla_{e_1}A_\xi)U\>-\< A_\xi V,\nabla_U e_1\>
\\&=-\<V,\nabla_{e_1}(A_\xi U)\>+\<V,A_\xi(\nabla_{e_1}X)\>
\\&=0\end{aligned}\end{equation}
for $U,V$ in $\mathcal D_2$. Hence we find $\<\nabla_U V,e_1\>=0$. Similarly, we have $\<\nabla_U V,e_2\>=0$. After combining these with \e{7.19.1}, we conclude that each leave  of $\mathcal D_2$ is a totally real totally geodesic submanifold of ${\bf C}^3$. Consequently, each connected component of $W$ is a  totally real 3-ruled minimal submanifold of ${\bf C}^3$.
If $W$ is dense in $N$,  we have the same conclusion by continuity. 

If $W$ is not dense in $N$, then the interior of  each connected component of $N-W$  is a totally geodesic real hypersurface of ${\bf C}^3$, which is obviously totally real 3-ruled. Consequently, by continuity the whose $N$ is a totally real 3-ruled minimal submanifold.
\end{proof}

For real hypersurfaces in $CP^3(4)$, we have

\begin{proposition}\label{P:7.1} If $N$ is a real hypersurface of $CP^3(4)$, then  we have
\begin{align} \label{7.20} Ric(J\xi,J\xi)\leq \frac{25}{2}H^2+ 4.\end{align}

The equality sign of \e{7.20} holds identically if and only if locally there exists an ortho\-normal frame $\{e_1,e_2,e_3=Je_1,e_4=Je_{2},e_5\}$ such that
\begin{equation}\begin{aligned}  &\sigma(e_1,e_1)=\lambda \xi,\;   \sigma(e_{2},e_{2})=- \lambda\xi,\;\\&\sigma(e_{3},e_{3})=\frac{1}{\lambda} \xi,\; \sigma(e_{4},e_{4})=- \frac{1}{\lambda}\xi,
 \\&\sigma(e_a,e_b)=0\;\, otherwise ,
\end{aligned}\end{equation} where $\lambda$ is a nowhere zero function.
\end{proposition} 
\begin{proof} Follows from Theorem \ref{T:6.1} and Lemma \ref{L:7.3}.
\end{proof}

Similarly, we also have the following result for real hypersurfaces in $CH^3(-4)$ by Theorem \ref{T:6.1} and Lemma \ref{L:7.3}.

\begin{proposition}\label{P:7.2} If $N$ is a real hypersurface of $CH^3(-4)$, then  we have
\begin{align} \label{7.22} Ric(J\xi,J\xi)\leq \frac{25}{2}H^2-4.\end{align}

The equality sign of \e{7.22} holds identically if and only if locally there exists an ortho\-normal frame $\{e_1,e_2,e_3=Je_1,e_4=Je_{2},e_5\}$ on $N$ such that
\begin{equation}\begin{aligned}  &\sigma(e_1,e_1)=\lambda \xi,\;   \sigma(e_{2},e_{2})=- \lambda\xi,\;\\&\sigma(e_{3},e_{3})=-\frac{1}{\lambda} \xi,\; \sigma(e_{4},e_{4})= \frac{1}{\lambda}\xi,
 \\&\sigma(e_a,e_b)=0\;\, otherwise ,
\end{aligned}\end{equation} where $\lambda$ is a nowhere zero function.
\end{proposition}

 Proposition \ref{P:7.1} and Proposition \ref{P:7.2} imply immediately the following.
 
 \begin{corollary} Every real hypersurface of $CP^3(4)$ $($resp., of $CH^3(-4))$ satisfying the equality case of \e{7.20} $($resp., the equality case of  \e{7.22}$)$  is $\delta(2,2)$-ideal  in the sense of \cite{C00,book2011}.
 \end{corollary}

\end{document}